\documentclass[a4paper,12pt]{article}
\usepackage{amssymb}
\usepackage{amsthm}
\usepackage{tikz,url}
\usepackage{tikz,pgfplots}
\usetikzlibrary{decorations.markings}
\usepackage{amsmath,amssymb}
\usepackage{todonotes}

\usepackage{color}
\usepackage{graphicx}
\usepackage{placeins,caption}

\DeclareMathOperator{\gp}{gp}
\DeclareMathOperator{\diam}{diam}

\usepackage[top=3cm,bottom=3cm,left=2.8cm,right=2.8cm]{geometry}
\newtheorem{theorem}{Theorem}[section]
\newtheorem{lemma}[theorem]{Lemma}
\newtheorem{corollary}[theorem]{Corollary}
\newtheorem{proposition}[theorem]{Proposition}

\newtheorem{conjecture}[theorem]{Conjecture}
\newtheorem{problem}[theorem]{Problem}

\theoremstyle{definition}

\newtheorem{definition}[theorem]{Definition}

\newcommand {\gpg} {{\rm gp}_{\rm g}}

\tikzset{middlearrow/.style={
		decoration={markings,
			mark= at position 0.75 with {\arrow[scale=2]{#1}} ,
		},
		postaction={decorate}
	}
}

\tikzset{midarrow/.style={
		decoration={markings,
			mark= at position 0.75 with {\arrow[scale=2]{#1}} ,
		},
		postaction={decorate}
	}
}

\begin{document}
	
	\title{Builder-Blocker General Position Games}
	\author{Sandi Klav\v{z}ar $^{a,b,c}$ \\ \texttt{\footnotesize sandi.klavzar@fmf.uni-lj.si} \and Jing Tian $^{d}$ \\ \texttt{\footnotesize jingtian526@126.com}  \and
		James Tuite $^{e}$ \\ \texttt{\footnotesize james.t.tuite@open.ac.uk}}
	
	\maketitle

	\noindent
	$^{a}$  Faculty of Mathematics and Physics, University of Ljubljana, Slovenia\\
	$^{b}$  Institute of Mathematics, Physics and Mechanics, Ljubljana, Slovenia \\
	$^{c}$  Faculty of Natural Sciences and Mathematics, University of Maribor, Slovenia\\
	$^{d}$  School of Science, Zhejiang University of Science and Technology, Hangzhou, Zhejiang 310023, PR China\\
	$^e$ School of Mathematics and Statistics, Open University, Milton Keynes, UK
	
	\begin{abstract}
		This paper considers a game version of the general position problem in which a general position set is built through adversarial play. Two players in a graph, Builder and Blocker, take it in turns to add a vertex to a set, such that the vertices of this set are always in general position. The goal of Builder is to create a large general position set, whilst the aim of Blocker is to frustrate Builder's plans by making the set as small as possible. The game finishes when no further vertices can be added without creating three-in-a-line and the number of vertices in this set is the game general position number. We determine this number for some common graph classes and provide sharp bounds, in particular for the case of trees. We also discuss the effect of changing the order of the players.  
	\end{abstract}
	
	
	\medskip\noindent
	{\bf AMS Subj.\ Class.\ (2020)}: 05C12, 05C57, 05C69
	
	\section{Introduction}
	
	Given a special property of subsets of a graph, it is often of practical and theoretical interest to ask for optimal sets with the desired property that are produced as a result of adversarial play. For example, Martin Gardner suggested such an approach to the well known chromatic number in~\cite{Gardner} and this problem now has an extensive literature (see~\cite{BarGryKierZhu} for a survey). A positional game in which the first player tries to build a large clique and the second player aims to frustrate the efforts of the first player was studied in~\cite{ErdSelf}. The domination game, in which one player, called Dominator, aims to build a small dominating set, whilst the second player, Staller, tries to keep the set as large as possible, was defined in~\cite{bresar-2010} (see also the book~\cite{book-2021}). Inspired by the latter game, in this paper we study a game version of the general position problem. 
	
	The general position problem originated in the no-three-in-line puzzle of Dudeney~\cite{dudeney-1917} and was generalised to the setting of graph theory in~\cite{ullas-2016,Manuel-2018}. A subset $S$ of the vertex set of a graph $G$ is in \emph{general position} if no shortest path of $G$ passes through more than two vertices of $S$. The \emph{general position number} $\gp (G)$ of $G$ is the number of vertices in a largest general position set. In the short period since this invariant was introduced, it has already been very well researched, see for example~\cite{AnaChaChaKlaTho, irsic-2024, KlaKriTuiYer, KorzeVesel, Patkos-2019, thomas-2024, tian-2021, yao-2022}. Additional research has also been carried out on edge general position sets~\cite{klavzar-2023, manuel-2022}. Building on another question of Gardner~\cite{Gardnerlower}, the recent paper~\cite{StefKlaKriTui} dealt with the smallest maximal general position sets of a graph, which represent the worst-case output of a greedy search for general position sets; the number of vertices in such a set is the \emph{lower general position number} $\gp^-(G)$. 
	
	The following general position game is introduced in~\cite{klavzar-2022}. Two players A and B take it in turns to select free vertices of a graph $G$ such that at any time the set of selected vertices is in general position; the last player that can move wins (the scenario in which the last player to move loses is considered in~\cite{ullas-2023+}). At the end of this game the resulting general position set will be maximal. This suggests that positional games may be a fruitful method of constructing maximal general position sets with order intermediate between $\gp ^{-}(G)$ and $\gp (G)$. For example, if $r_1 \geq\dots \geq r_t$, then $\gp (K_{r_1,\dots,r_t}) = \max \{ r_1,t\}$, $\gp ^- (K_{r_1,\dots ,r_t}) = \min \{ r_t,t\} $ and if $t$ is odd and there is a part in the partition with odd order, then the optimal play described in~\cite{klavzar-2022} produces a maximal general position set of order $t$. However, as the players of this game are not concerned with the length of their game, just in who moves last, optimal game play can produce maximal general position sets of different orders. In order to give a well-defined invariant associated with these adversarial games, we define the following game.
	
	\begin{definition}[Builder-Blocker general position game]
		Builder and Blocker take it in turns to choose an unmarked vertex of a graph $G$. If Builder moves first we speak of a {\em Builder general position game}, otherwise of a {\em Blocker general position game}. For brevity we will simply call the first game a {\em B-game} and the second a {\em B'-game}. At each stage the set of marked vertices must be in general position. The game ends when no further vertices can be selected. The goal of Builder is to produce a largest possible general position set, whilst the aim of Blocker is to frustrate Builder by forcing them to build a general position set containing as few vertices as possible.
	\end{definition}
	
	\begin{definition}
		The number of vertices in the general position set built by optimum play in a B-game on a graph $G$ is the \emph{Builder-game general position number} (\emph{B-game general position number}) of $G$, which we will denote by $\gpg(G)$. The corresponding invariant for the B'-game is the \emph{Blocker-game general position number} (\emph{B'-game general position number}) of $G$, which we will denote by $\gpg'(G)$.
	\end{definition}
	
	The order of a graph $G$ will be denoted by $n(G)$. For a positive integer $k$, the set $\{1,\ldots, k\}$ will be written $[k]$. The distance $d_G(u,v)$ between vertices $u$ and $v$ of a graph $G$ is the length of a shortest $u,v$-path in $G$. The largest value of $d_G(u,v)$ over all pairs $u,v \in V(G)$ is the \emph{diameter} $\diam (G)$ of $G$.
	For $0 \leq t \leq \diam (G)$, the set of vertices at distance exactly $t$ from $u \in V(G)$ is $N^t(u)$; in particular, the \emph{neighbourhood} $N(u) = N^1(u)$ of $u \in V(G)$ is the set $\{ v \in V(G): u \sim v\} $. The \emph{degree} of $u \in V(G)$ is $\deg_G(u) = |N(u)|$ and the largest degree in the graph is the \emph{maximum degree} $\Delta (G)$. A vertex $u$ is a \emph{leaf} of $G$ if $\deg_G(u) = 1$; the number of leaves of $G$ will be written $\ell (G)$.  
	
	The rest of the paper is structured as follows. In Section~\ref{sec:bounds} we give some useful bounds on the game general position numbers and determine these numbers for several classes of graphs, including the Kneser graphs $K(n,2)$. In Section~\ref{sec:number=2} we explore graphs with small game general position numbers, showing the connection with universal lines and characterising graphs $G$ with $\gpg'(G) = 2$. The main message of Section~\ref{Sec:player order} is that the order of the players in the general position game is important, since both $\gpg '(G)-\gpg(G)$ and $\gpg (G)-\gpg '(G)$ can be arbitrarily large. Finally, in Section~\ref{sec:trees}, we derive a sharp upper bound for trees in the B'-game and characterise the equality case.  
	
	\section{Bounds and exact values}
	\label{sec:bounds}
	In this section, we provide some bounds on the game general position numbers that will prove useful in the remainder of the paper. We then discuss the results of the B-game and the B'-game on some common graph classes, including complete multipartite graphs and Kneser graphs.
	
	As both games, the B-game and the B'-game, result in a maximal general position set, we have the trivial bounds 
	\begin{align}
		\gp ^-(G) & \leq \gpg(G) \leq \gp (G), \label{eq:for-B-game}  \\     
		\gp ^-(G) & \leq \gpg'(G) \leq \gp (G). \label{eq:for-B'-game} 
	\end{align}

	The following lemma will be very useful in our investigations of the Builder-Blocker general position game. 
	
	\begin{lemma}
		\label{lem:useful}
		Let $G$ be a graph and let $S$ be the set of vertices played in a B-game or B'-game so far. Let $S'$ be the set of vertices that are playable as the next move of the game. If $S\cup S'$ is a general position set, then the game will finish precisely after all the vertices from $S'$ have been played.
	\end{lemma}
	
	\begin{proof}
		If $u$ is a vertex of $V(G)\setminus (S\cup S')$, then $u$ cannot be chosen by either Builder or Blocker in the current state of play; hence $S \cup \{ u\} $ is not a general position set, i.e., there is a shortest path in $G$ containing $u$ and two vertices of $S$. This will obviously remain true at any future stage of the game, so that the vertex $u$ cannot be played at any point in the game once the vertices of $S$ have been chosen. On the other hand, since $S\cup S'$ is a general position set, playing the vertices from $S'$ one by one is legal until all the vertices from $S'$ have been played. 
	\end{proof}
	
	If $G$ is a graph and $uv\in E(G)$, then let 
	\begin{align*}
		W_{uv} & = \{w\in V(G):\ d_G(u,w) < d_G(v,w)\}\,, \\
		W_{vu} & = \{w\in V(G):\ d_G(v,w) < d_G(u,w)\}\,, \\
		_uW_v & = \{w\in V(G):\ d_G(u,w) = d_G(v,w)\}\,. 
	\end{align*}
	Note that $u\in W_{uv}$, $v\in W_{vu}$, and that $V(G)$ is the disjoint union of $W_{uv}$, $W_{vu}$ and $_uW_v$. For an example of these sets consider the Petersen graph $P$ in  Fig.~\ref{fig:Petersen}. 
	
	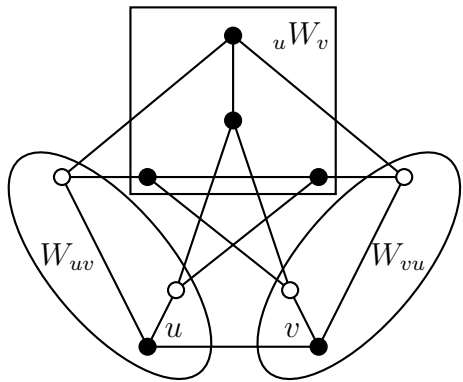
\begin{figure}[ht!]
		\begin{center}
			\begin{tikzpicture}[scale=1.5,style=thick]
				\tikzstyle{every node}=[draw=none,fill=none]
				\def\vr{2pt} 
				
				\begin{scope}[yshift = 0cm, xshift = 0cm]
					
					\path (0.25,0) coordinate (x1);
					\path (1.75,0) coordinate (x2);
					\path (2.5,1.5) coordinate (x3);
					\path (1,2.75) coordinate (x4);
					\path (-0.5,1.5) coordinate (x5);
					\path (0.5,0.5) coordinate (x6);
					\path (1.5,0.5) coordinate (x7);
					\path (1.75,1.5) coordinate (x8);
					\path (1,2) coordinate (x9);
					\path (0.25,1.5) coordinate (x10);
					
					\draw (x1)--(x2)--(x3) -- (x4)--(x5)--(x1);
					\draw (x6)--(x8)--(x10)--(x7)--(x9)--(x6);
					\draw (x1)--(x6);
					\draw (x2)--(x7);
					\draw (x3)--(x8);
					\draw (x4)--(x9);
					\draw (x5)--(x10);
					\draw[rotate=40] (0.4,0.6) ellipse (15pt and 35pt);
					\draw[rotate=-40] (1.15,1.9) ellipse (15pt and 35pt);
					\draw[draw=black] (0.1,1.35) rectangle (1.9,3.0);
					\draw (x1)  [fill=black] circle (\vr);
					\draw (x2)  [fill=black] circle (\vr);
					\draw (x3)  [fill=white] circle (\vr);
					\draw (x4)  [fill=black] circle (\vr);
					\draw (x5)  [fill=white] circle (\vr);
					\draw (x6)  [fill=white] circle (\vr);
					\draw (x7)  [fill=white] circle (\vr);
					\draw (x8)  [fill=black] circle (\vr);
					\draw (x9)  [fill=black] circle (\vr);
					\draw (x10)  [fill=black] circle (\vr);

					\draw (x1)++(0.23,0.15) node {$u$};
					\draw (x2)++(-0.23,0.15) node {$v$};
					\draw (x1)++(-0.7,0.8) node {$W_{uv}$};
					\draw (x2)++(0.7,0.8) node {$W_{vu}$};
					\draw (x4)++(0.6,0.0) node {$_uW_v$};
					
				\end{scope}
			\end{tikzpicture}
		\end{center}
		\caption{Sets $W_{uv}$, $W_{vu}$, and $_uW_v$ in the Petersen graph}
		\label{fig:Petersen}
	\end{figure}
	
	\begin{theorem}
		\label{thm:bound-with-uWv}
		If $G$ is a graph with $n(G)\ge 2$, then 
		$$2\le \gpg(G) \le 2 + \max_{u\in V(G)} \min_{v\in N(u)} |_uW_v|\,.$$
	\end{theorem}
	
	\begin{proof}
		The lower bound is clear. Consider now the B-game played on $G$ and assume that Builder selected a vertex $u$ as their first move. Suppose that Blocker replies by playing a neighbour $v$ of $u$. If $w\in W_{uv}$, then $d_G(w,v) = d_G(w,u)+1$, and as $uv\in E(G)$, we infer that $w$, $u$, and $v$ lie on a common shortest path. It follows that in the rest of the game no vertex from $W_{uv}$ will be selected. By the same argument, no vertex from $W_{vu}$ will be selected in the rest of the game; hence at most $2 + |_uW_v|$ vertices will be selected. Thus if Builder starts by playing $u$, then Blocker has a strategy that limits the resulting general position set to at most $2 + \min_{v\in N(u)} |_uW_v|$ vertices. As Builder is the first to play and wishes to maximise the number of vertices selected, the upper bound follows. 
	\end{proof}
	
	Theorem~\ref{thm:bound-with-uWv} instantly implies the following useful corollary.
	
	\begin{corollary}
		\label{cor:bound-with-uWv}
		If $G$ is a graph with $n(G)\ge 2$, then 
		$$2\le \gpg(G) \le 2 + \max\{|_uW_v|:\ uv\in E(G)\}\,.$$    
	\end{corollary}
	
	The upper bound of Corollary~\ref{cor:bound-with-uWv} (and thus also of Theorem~\ref{thm:bound-with-uWv}) is sharp, as demonstrated by complete graphs and the Petersen graph $P$. If $uv$ is an edge of $K_n$, then we have $|_uW_v| = n-2$, so that the bound yields $\gpg(K_n)\le n$. As for the Petersen graph, Theorem~\ref{thm:bound-with-uWv} yields $\gpg(P)\le 6$, see Fig.~\ref{fig:Petersen} again. On the other hand, if Builder first plays some vertex $u$, then either Blocker replies with a neighbour $v$ of $u$, or Blocker chooses a vertex $v$ with $d_G(u,v) = 2$, which case Builder can choose a vertex of $N(u)\setminus N(v)$ on their next turn; in either case, Lemma~\ref{lem:useful} shows that the resulting set contains six vertices and hence $\gpg(P) = 6$. 
	
	If $G$ is a connected, bipartite graph, then $_uW_v = \emptyset$ for every edge $uv$ of $G$, which implies the following: 
	
	\begin{corollary}
		\label{cor:bipartite}
		If $G$ is a connected, bipartite graph with $n(G)\ge 2$, then $\gpg(G)=2$. 
	\end{corollary}
	
	As we will see in Theorem~\ref{the:buildblock multipartite}, $\gpg'(G)$ can be arbitrary large even in the class of bipartite graphs; hence Theorem~\ref{thm:bound-with-uWv} does not extend to the B'-game. However, we can still obtain a useful bound for the following family of graphs: let ${\cal G}$ be the set of graphs $G$ such that for any vertices $x, y\in V(G)$ with $d_G(x,y)\ge 2$, there exists an edge $yz$, such that $x$ is equidistant from $y$ and $z$. This family ${\cal G}$ is quite large. For instance, it contains $C_5$, the Petersen graph, complete graphs, many circulants, and wheels $W_k$, $k\ge 5$. 
	
	\begin{theorem}
		\label{thm:bound-with-uWv-for-B'-game}
		If $G\in {\cal G}$, then 
		$$\gpg'(G) \le 2 + \max\{|_xW_y|:\ xy\in E(G)\}\,.$$
	\end{theorem}
	
	\begin{proof}
		Consider the B'-game and let $u$ and $v$ be the first two moves of Blocker and Builder, respectively. If $uv\in E(G)$, then the same argument as given in the proof of Theorem~\ref{thm:bound-with-uWv} gives $\gpg'(G) \le 2 + \max\{|_xW_y|:\ xy\in E(G)\}$. Assume hence that $d_G(u,v)\ge 2$. By assumption, there exists a vertex $w$ such that $vw\in E(G)$ and $d_G(u,v) = d_G(u,w)$. It follows that $w$ is a legal second move of Blocker. After these three moves have been played, only vertices from $_vW_w$ will be played in the rest of the game. Since $u\in\ _vW_w$, we again have $\gpg'(G) \le 2 + \max\{|_xW_y|:\ xy\in E(G)\}$.
	\end{proof}
	
	The upper bound in Theorem~\ref{thm:bound-with-uWv-for-B'-game} is attained by $C_5$, the Petersen graph and complete graphs. 
	
	To demonstrate that the class of graphs ${\cal G}$ can be large, consider the following. Let ${\cal G}_2$ be the set of graphs $G$ such that for any vertices $x, y\in V(G)$ with $d_G(x,y) = 2$ there exists an edge $yz$, such that $x$ is equidistant from $y$ and $z$. Recall that the \emph{lexicographic product} $G\circ H$ of graphs $G$ and $H$ is the graph with the vertex set $V(G)\times V(H)$, vertices $(g,h)$ and $(g',h')$ being adjacent if either $gg'\in E(G)$ or $g=g'$ and $hh'\in E(H)$. Then we have: 
	
	\begin{proposition}
		\label{prop:lexicographic}
		If $G$ is a connected graph with $n(G)\ge 2$ and $H$ is a graph from ${\cal G}_2$, then $G\circ H \in {\cal G}$.
	\end{proposition}
	
	\begin{proof}
		Let $(g,h)$ and $(g',h')$ be arbitrary vertices of $G\circ H$ with $d_{G\circ H}((g,h), (g',h')) \ge 2$. We distinguish two cases. 
		
		Assume first that $g=g'$. Then $d_{G\circ H}((g,h), (g',h')) = 2$ by the structure of the lexicographic product. In the first subcase, assume that $d_{H}(h,h') = 2$. Since $H\in {\cal G}_2$, there exists a neighbour $h''$ of $h'$ in $H$ such that $d_{H}(h,h'') = 2$. Then also $d_{G\circ H}((g,h), (g,h'')) = 2$. In the second subcase, assume that $d_{H}(h,h') \ge 3$. Now let $h''$ be a neighbour of $h'$ such that $d_{H}(h,h'') = d_{H}(h,h') - 1$. Then we have $d_{G\circ H}((g,h), (g,h')) = 2 = d_{G\circ H}((g,h), (g,h''))$. 
		
		Assume second that $g\ne g'$. Let $h''$ be an arbitrary neighbour of $h$ in $H$. Then we have $d_{G\circ H}((g,h), (g',h')) = d_{G\circ H}((g,h), (g',h''))$ and we are done.
	\end{proof}
	
	Clearly, ${\cal G}\subseteq {\cal G}_2$. Moreover, ${\cal G}$ is a proper subset of ${\cal G}_2$ as the generalised Petersen graph $P(10,2)$ (alias the dodecahedral graph) demonstrates; for the definition of generalised Petersern graphs see, for instance,~\cite[p.~20]{Bondy-2008}. Indeed, since each pair of vertices of $P(10,2)$ which is at distance $2$ lies on a common (isometric) $5$-cycle, we see that $P(10,2)\in {\cal G}_2$. On the other hand, if $u$ and $v$ are two antipodal vertices of $P(10,2)$, then $d_{P(10,2)}(u,v)=5$, but for any neighbour $w$ of $v$ we have $d_{P(10,2)}(u,w)=4$. We conclude that $P(10,2)\notin {\cal G}$. 
	
	Trivially, $\gpg(K_n) = n$ for $n\ge 1$. It is also straightforward to see that if $n \geq 6$ is even, then $\gpg(C_n) = 2$ and $\gpg'(C_n) = 3$, while if $n$ is odd, then $\gpg(C_n) = \gpg'(C_n) = 3$.
	
	\begin{theorem}\label{the:buildblock multipartite}
		If $t\ge 2$ and $r_1 \geq \cdots \geq r_t\ge 2$, then 
		\[ \gpg(K_{r_1,\dots,r_t}) = \min \{ r_1,t\} \]
		and
		\[ \gpg '(K_{r_1,\dots,r_t}) = \max \{ r_t,t\}.\]
	\end{theorem}
	\begin{proof}
		Assume that Builder starts with the move $a_1$ in a part $X$ of order $|X| = r_i$.  
		
		If Blocker can select their second vertex from  $X \setminus \{ a_1\} $, then all subsequent vertices chosen must lie in $X$, so that optimal play produces the maximal general position set $X$. If Blocker chooses their first vertex in a part $X'$ not equal to $X$, then each subsequent vertex chosen must lie in a new part, so that the game produces a clique with one vertex from each part, yielding a general position set of order $t$. Hence Blocker can limit Builder to a general position set of order $\min \{ r_i,t\} $. Hence it is in Builder's interests to choose their first vertex in a part of order $r_1$. Thus the general position set produced by optimal play has order $\min \{ r_1,t\} $.
		
		A similar argument shows that if Blocker's first move in part $X$ of order $r_i$, then Builder has the choice of either moving within $X$ (resulting in a general position set of order $r_i$), or selecting a vertex in a new part $X' \not = X$, in which case the game builds a clique of order $t$. Therefore the order of the general position set produced is $\max \{ r_i,t\} $ and, in order to minimise this, Blocker must start in a partite set of order $r_t$.
	\end{proof}
	
	The result for complete multipartite graphs allows us to prove a realisation result for the game general position number versus the $\gp $-number and lower $\gp $-number as a corollary.
	
	\begin{corollary}
		For any integers $a,b,c$ with $2 \leq c \leq b \leq a$ there exists a graph $G$ with $\gp (G) = a$, $\gpg(G) = b$ and $\gp ^-(G) = c$.
	\end{corollary}
	\begin{proof}
		Recall that $\gp (K_{r_1,\dots ,r_t}) = \max \{ r_1,t\} $ and $\gp ^- (K_{r_1,\dots ,r_t}) = \min \{ r_t,t\} $   Hence the result follows by applying Theorem~\ref{the:buildblock multipartite} to a complete multipartite graph $K_{r_1,\dots,r_t}$ with $r_1 = a$, $t = b$ and $r_t = c$. 
	\end{proof}
	
	We now determine the result of the general position game on Kneser graphs $K(n,2)$. Recall that the vertices of this graph are the subsets of order two of $[n]$ and two subsets are adjacent if and only if they are disjoint  sets. The general position number of $K(n,2)$ was shown to be $n-1$ for $n \geq 7$ in~\cite{Knesergraph} and the lower general position number was determined in~\cite{StefKlaKriTui} as follows.
	
	\begin{theorem}\label{thm:lowergp Kneser}
		{\rm \cite[Theorem 13]{StefKlaKriTui}}
		The lower gp-number of the Kneser graph $K(n,2)$ is 
		\[ \gp ^- (K(n,2))=\begin{cases}
			3, &\text{if $n = 3$, } \\
			6, &\text{if $n = 4$, } \\
			4, &\text{if $n = 5$, } \\
			\left \lfloor \frac{n}{2} \right \rfloor , & \text{if $6 \leq n \leq 11$, } \\
			6, & \text{if $n \geq 12$.}
		\end{cases}\]
	\end{theorem}
	It turns out that for $n \geq 12$ the game general position numbers coincide with the lower general position number.
	
	\begin{theorem}
		\label{thm:Kneser}
		For $n \geq 4$ we have $\gpg(K(n,2)) = \gpg '(K(n,2)) = 6$.
	\end{theorem}
	\begin{proof}
		We use the description of the maximal general position sets from Theorem~\ref{thm:lowergp Kneser}. We first consider the B-game played on $K(n,2)$ and assume without loss of generality that Builder starts with the move $a_1 = \{ 1,2\}$.  Assume further that Blocker's first move is a 2-subset disjoint from $a_1$, say $b_1 = \{ 3,4\}$.  If Builder's next move is a 2-subset $a_2$ of  $\{ 1,2,3,4\} $, then the general position set built during the game is the set of six 2-subsets of $[4]$. Otherwise if $a_2$ is disjoint with $\{ 1,2,3,4\} $ then Builder and Blocker must build a clique of order $\left \lfloor \frac{n}{2} \right \rfloor $. Hence for $n \leq 11$ Builder would choose a 2-subset of $[4]$, for $n \in \{12,13\}$ Builder's choice is arbitrary, and for $n \geq 14$ Builder can choose $\{ 5,6\} $. 
		
		Suppose that Blocker's first move is $b_1 = \{ 1,3\} $. If Builder picks $a_2 = \{ 1,i\} $ for some $4 \leq i \leq n$, then Blocker can either pick a set $b_2 = \{ 1,j\} $,  $j \in [n]\setminus\{ 2,3,i\} $, in which case they build a general position set of order $n-1$, or else Blocker can pick a 2-subset of $\{ 1,2,3,i\} $, in which case they produce a general position set of order six. Therefore for $4 \leq n \leq 6$ Blocker will choose the first option and for $n \geq 7$ Blocker can choose the second. If Builder's second move $a_2$ contains 2 or 3, then they will build a general position set of order six. Hence if $n \geq 7$ then either option for $a_2$ will result in a general position set of order six and for $4 \leq n \leq 6$ Builder will choose a set containing 2 or 3, again leading to a general position set of order six.
		
		In summary, if Blocker chooses $b_1 =\{ 1,3\} $, then the game gives a set of order six, whereas if $b_1 = \{3,4\} $, then the game will give a general position set of order six for $4 \leq n \leq 11$ and order $\left \lfloor \frac{n}{2} \right \rfloor $ for $n \geq 12$. Thus the best that Blocker can do is restrain Builder to a general position set of order six.
		
		A similar argument establishes the value of $\gpg '(K(n,2))$.
	\end{proof}
	
	\section{Graphs with small game general position numbers}
	\label{sec:number=2}
	
	Since $\gpg(G)\ge 2$ and $\gpg'(G)\ge 2$ hold for any non-trivial graph $G$, it is natural to ask when equality can hold. We now proceed to answer this question. 
	
	Let $M = (X, d_M)$ be an arbitrary metric space and $x,y\in X$. Then the {\em line} ${\cal L}_M(x,y)$ induced by $x$ and $y$ is the following set of points from $M$:  
	$$\{z\in X:\ d_M(x,y) = d_M(x,z) + d_M(z,y)\ {\rm or}\ d_M(x,y) = |d_M(x,z) - d_M(z,y)|\}\,.$$
	The line ${\cal L}_M(x,y)$ is {\em universal} if it contains the whole set $X$. Considering a graph $G$ as a metric space, these definitions transfer directly to $G$, see~\cite{Rodriguez-2022}. We can now describe the graphs $G$ with $\gpg(G) = 2$ by the following result which generalises Corollary~\ref{cor:bipartite}. 
	
	\begin{lemma}
		\label{lem:universal-line}
		If $G$ is a graph with $n(G)\ge 2$, then $\gpg(G) = 2$ if and only if each vertex of $G$ is contained in a pair that induces a universal line. In particular, if $G$ is a vertex-transitive graph with $\gp ^-(G) = 2$, then $\gpg(G) = 2$. Moreover, $\gpg '(G) = 2$ if and only if there exists a vertex $u \in V(G)$ such that for any $v \in V(G) \setminus \{ u\} $ the pair $\{ u,v\} $ induces a universal line.
	\end{lemma}
	
	\begin{proof}
		Assume that $\gpg(G) = 2$. Let $x$ be an arbitrary vertex of $G$ and consider the B-game in which Builder selects $x$ as the first move. Since $\gpg(G) = 2$, there exists a reply $x'$ of Blocker such that the game is over after this move. This means that the line ${\cal L}_G(x,x')$ is universal. Conversely, assume that each vertex of $G$ is contained in a universal line. Consider the B-game and let $x$ be the first optimal move of Builder. By our assumption, there exists a vertex $x'$ such that the line ${\cal L}_G(x,x')$ is universal. Then Blocker replies with the move $x'$ and the game is over. 
		
		Assume that $G$ is vertex-transitive with $\gp ^-(G) = 2$. Let $\{x,x'\}$ be a maximal general position set of $G$. Then the line ${\cal L}_G(x,x')$ is universal. If $y$ is an arbitrary vertex of $G$, let $\alpha$ be an automorphism of $G$ with $\alpha(x) = y$. Then the line ${\cal L}_G(\alpha(x),\alpha(x')) = {\cal L}_G(y,\alpha(x'))$ is a universal line containing $y$. As above, $\gpg(G) = 2$.
		
		The observation for graphs with $\gpg '(G) = 2$ follows similarly.
	\end{proof}
	
	For example, cocktail-party graphs $K_{2,\ldots, 2}$ are vertex-transitive with lower general position number $2$, hence $\gpg(K_{2,\ldots, 2}) = 2$. 
	
	We now turn our attention to graphs $G$ with $\gpg'(G) = 2$. To this end we introduce the family of graphs ${\cal H}$ containing all graphs of order at least $2$ that are constructed in the following way. Let $G_i$, $i\in [k]$, be a collection of $k\ge 0$ complete bipartite graphs $G_i = K_{2,n_i}$, $n_i \ge 2$, and in each graph $G_i$ let the partite set with two vertices be $X_i = \{ u_i,u_i'\} $ (if $n_i = 2$, then obviously the choice is arbitrary). Then identify all the vertices $u_i$, $i\in [k]$, into a single vertex $u$. Next, add a pendant path of arbitrary length (possibly $0$) to each of the vertices $u_i'$. Finally, attach an arbitrary number (possibly $0$) of paths at $u$. An example of a graph belonging to the family ${\cal H}$ is shown in Fig.~\ref{fig:gpg=2}.  
	
	\begin{figure}[ht!]
		\centering
		\begin{tikzpicture}[x=0.2mm,y=-0.2mm,inner sep=0.2mm,scale=0.7,thick,vertex/.style={circle,draw,minimum size=10}]
			\node at (0,0) [vertex] (v1) {$u$};
			\node at (100,-100) [vertex] (v2) {};
			\node at (100,-50) [vertex] (v3) {};
			\node at (100,0) [vertex] (v4) {};
			
			\node at (100,50) [vertex] (v5) {};
			\node at (100,100) [vertex] (v6) {};
			
			\node at (200,-50) [vertex] (v7) {};
			\node at (200,75) [vertex] (v8) {};
			
			\node at (300,-50) [vertex] (v9) {};
			\node at (400,-50) [vertex] (v10) {};
			\node at (500,-50) [vertex] (v11) {};
			
			\node at (-100,-75) [vertex] (v12) {};
			\node at (-100,-25) [vertex] (v13) {};
			\node at (-100,25) [vertex] (v14) {};
			\node at (-100,75) [vertex] (v15) {};
			
			\node at (-200,-75) [vertex] (v16) {};
			\node at (-300,-75) [vertex] (v17) {};
			\node at (-200,-25) [vertex] (v18) {};

			\path

			(v1) edge (v2)
			(v1) edge (v3)
			(v1) edge (v4)
			(v1) edge (v5)
			(v1) edge (v6)
			
			(v7) edge (v2)
			(v7) edge (v3)
			(v7) edge (v4)
			
			(v8) edge (v5)
			(v8) edge (v6)
			
			(v7) edge (v9)
			(v9) edge (v10)
			(v10) edge (v11)
			
			(v1) edge (v12)
			(v1) edge (v13)
			(v1) edge (v14)
			(v1) edge (v15)
			
			(v12) edge (v16)
			(v16) edge (v17)
			(v13) edge (v18)
			
			;
		\end{tikzpicture}
		\caption{A graph from the family ${\cal H}$}
		\label{fig:gpg=2}
	\end{figure}
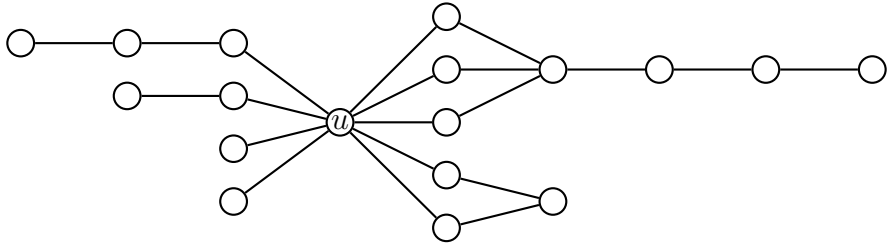

	\begin{theorem}
		\label{thm:gpg'=2}
		For any graph $G$, $\gpg '(G) = 2$ if and only if $G\in {\cal H}$. 
	\end{theorem}
	
	\begin{proof}
		Let $G\in {\cal H}$ and let $u$ be the vertex of $G$ as defined in the description of the class ${\cal H}$. It is then straightforward to check that for any vertex $v \neq u$, the pair $\{ u,v\} $ induces a universal line, so by Lemma~\ref{lem:universal-line} we have $\gpg '(G) = 2$ (i.e., if Blocker starts with $u$, then whichever vertex Builder selects for their next vertex, the game is over).
		
		Conversely, let $G$ be a arbitrary graph with $\gpg '(G) = 2$. Let $u$ be Blocker's first move selected optimally. Suppose that there is a vertex $v$ in $N^t(u)$ ($1 \leq t \leq \diam (G)-1$) such that $v$ has at least two neighbours $w_1,w_2$ in $N^{t+1}(u)$; then $\{ u,w_1,w_2\} $ is in general position and Builder could guarantee a general position set of order at least three by choosing $w_1$. Hence each vertex of $N^t(u)$ has at most one neighbour in $N^{t+1}(u)$.
		
		Suppose that $G$ contains a cycle, but that $u$ does not lie on any cycles. Let $C$ be a girth cycle of $G$, and $P$ be a shortest path from $u$ to $C$, terminating in a vertex $w \in V(C)$. Since a girth cycle in isometric, $u$ together with the two neighbours of $w$ on $C$ constitutes a general position set, so that if Builder plays a neighbour of $w$ on $C$, then Blocker is obliged to take a further move and $\gpg '(G) \geq 3$. Now suppose that $u$ lies on a cycle of $G$ and let $C$ be a shortest such cycle; as $\gp (C_n) = 3$ unless $n = 4$, if the length of $C$ does not equal four, then Builder can choose a vertex $v$ from $C$ such that $\{ v\} $ can be completed to a general position set of at least three vertices. 
		
		As $u$ does not lie on a triangle, $N(u)$ is independent. Let $W$ be the set of vertices in $N^2(u)$ that lie on a 4-cycle containing $u$. Since each vertex of $N(u)$ has at most one neighbour in $N^{2}(u)$, if $w_1,w_2 \in W$, then $N(w_1) \cap N(w_2) \cap N(u) =\emptyset $. Furthermore, each set $N^t(u)$ is an independent set, since otherwise $u$ would be contained in an odd cycle and considering a shortest such cycle Builder would again be able to force a general position set with at least three vertices. We also see that any vertex in $N^t(u)$, where $t \geq 3$, can have at most one neighbour in $N^{t-1}(u)$, for otherwise we could consider a vertex with smallest possible $t \geq 3$ and see that this would create an isometric odd cycle through $u$.
		
		From the above arguments we may conclude that any graph $G$ with $\gpg '(G) = 2$ belongs to ${\cal H}$. 
	\end{proof}
	
	In fact, every graph $G$ in the family ${\cal H}$ satisfies the first part of Lemma~\ref{lem:universal-line}, i.e., for every vertex $u \in V(G)$ there is a vertex $u'$ such that $\{ u,u'\} $ induces a universal line; therefore for each such graph we also have $\gpg (G) = 2$.
	
	\begin{corollary}\label{cor:realisation for two}
		If $\gpg '(G) = 2$, then $\gpg (G) = 2$.
	\end{corollary}
	
	\section{Who goes first?}
	\label{Sec:player order}
	
	In the domination game, it is known that changing the order of the players Dominator and Staller can make a difference of at most one to the order of the resulting dominating set~\cite{bresar-2010, kinnersley-2013}. Moreover, from the bounds and exact values that we have derived so far, it may appear that the general position set produced by the B-game is no larger than that produced by the B'-game; in particular Corollary~\ref{cor:realisation for two} shows that this is true if either of the numbers $\gpg (G)$ or $\gpg '(G)$ equal two. However, it turns out that the order of the players in the general position game is very important, as $\gpg (G)-\gpg' (G)$ and $\gpg '(G)-\gpg(G)$ can both be arbitrarily large. This feature thus strongly distinguishes general position games from domination games. 
	
	We start with the simple result, which characterises the pairs $(a,b)$ with $a \leq b$ such that there exists a graph $G$ with $\gpg (G) = a$ and $\gpg'(G) = b$.
	
	\begin{proposition}\label{prop:small builder num}
		For any integers $a, b$ such that $2 \leq a \leq b$, there exists a graph $G$ with $\gpg (G) = a$ and $\gpg '(G) = b$.
	\end{proposition}
	\begin{proof}
		If $a = b$, then the complete graph suffices, so we can assume that $a < b$. Then Theorem~\ref{the:buildblock multipartite} shows that the complete $a$-partite graph with each part of size $b$ will suffice. 
	\end{proof}
	
	Now we show that the general position number for the B-game can be arbitrarily larger than the general position number for the B'-game.
	
	\begin{theorem}\label{thm:largerbuildernum}
		The difference $\gpg(G) -\gpg'(G)$ can be arbitrarily large.
	\end{theorem}
	\begin{proof}
		Consider the graph $G(r,s)$ formed by gluing together $r$ copies of $C_4$ and $s$ triangles along an edge. For $1 \leq i \leq r$, we will label the vertices of the $i$-th copy of $C_4$ by $x,y,x_i,y_i$, where $x \sim x_i$, $y \sim y_i$, $x_i \sim y_i$ and $x \sim y$ and $xy$ is the edge that is identified in the $r$ copies. Similarly, we will denote the unidentified vertex of the $j$-th triangle by $w_j$ for $1 \leq j \leq s$. Let $X = \{ x_i : 1 \leq i \leq r\} $, $Y = \{ y_i : 1 \leq i \leq r\} $ and $W = \{ w_j : 1 \leq j \leq s\} $. An example of this construction is shown in Fig.~\ref{fig:larger gpg'}. We will take $r$ to be larger than $s$ and assume $s$ to be odd and at least three.
		
		\begin{figure}[ht!]
			\centering
			\begin{tikzpicture}[x=0.2mm,y=-0.2mm,inner sep=0.2mm,scale=1.5,thick,vertex/.style={circle,draw,minimum size=10}]
				\node at (-70,-60) [vertex]  (x-1) {};
				\node at (-60,-30) [vertex]  (x0) {};
				\node at (-50,0) [vertex]  (x1) {};
				\node at (-40,30) [vertex]  (x2) {};
				\node at (-30,60) [vertex]  (x3) {};
				\node at (-20,90) [vertex]  (x4) {};
				
				\node at (70,-60) [vertex]  (y-1) {};
				\node at (60,-30) [vertex]  (y0) {};
				\node at (50,0) [vertex]  (y1) {};
				\node at (40,30) [vertex]  (y2) {};
				\node at (30,60) [vertex]  (y3) {};
				\node at (20,90) [vertex]  (y4) {};
				
				\node at (-120,120) [vertex]  (x) {};
				\node at (120,120) [vertex]  (y) {};
				
				\node at (-60,170) [vertex]  (s1) {};
				\node at (-30,170) [vertex]  (s2) {};
				\node at (0,170) [vertex]  (s3) {};
				\node at (30,170) [vertex]  (s4) {};
				\node at (60,170) [vertex]  (s5) {};
				\path
				(x0) edge (y0)
				(x-1) edge (y-1)
				(x0) edge (x)
				(x-1) edge (x)
				(y0) edge (y)
				(y-1) edge (y)

				(x) edge (s1)
				(x) edge (s2)
				(x) edge (s3)
				(x) edge (s4)
				(x) edge (s5)
				
				(y) edge (s1)
				(y) edge (s2)
				(y) edge (s3)
				(y) edge (s4)
				(y) edge (s5)
				
				(x) edge (y)
				
				(x) edge (x1)
				(x) edge (x2)
				(x) edge (x3)
				(x) edge (x4)
				
				(y) edge (y1)
				(y) edge (y2)
				(y) edge (y3)
				(y) edge (y4)
				
				(x1) edge (y1)
				(x2) edge (y2)
				(x3) edge (y3)
				(x4) edge (y4)
				
				;
				\node [above=2.0mm] at (x-1) {$x_6$};
				\node [above=2.0mm] at (x0) {$x_5$};
				\node [above=2.0mm] at (x1) {$x_4$};
				\node [above=2.0mm] at (x2) {$x_3$};
				\node [above=2.0mm] at (x3) {$x_2$};
				\node [above=2.0mm] at (x4) {$x_1$};
				\node [above=2.0mm] at (y-1) {$y_6$};
				\node [above=2.0mm] at (y0) {$y_5$};
				\node [above=2.0mm] at (y1) {$y_4$};
				\node [above=2.0mm] at (y2) {$y_3$};
				\node [above=2.0mm] at (y3) {$y_2$};
				\node [above=2.0mm] at (y4) {$y_1$};
				\node [below=2.0mm] at (s1) {$w_1$};
				\node [below=2.0mm] at (s2) {$w_2$};
				\node [below=2.0mm] at (s3) {$w_3$};
				\node [below=2.0mm] at (s4) {$w_4$};
				\node [below=2.0mm] at (s5) {$w_5$};
				\node [left=2.0mm] at (x) {$x$};
				\node [right=2.0mm] at (y) {$y$};
				
			\end{tikzpicture}
			\caption{The graph $G(6,5)$}\label{fig:larger gpg'}
		\end{figure}
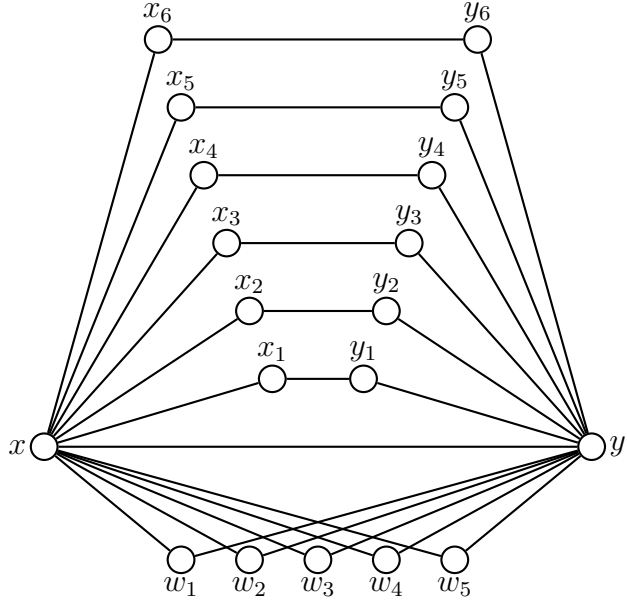

		We begin by classifying the maximal general position sets in $G(r,s)$. Let $S$ be such a maximal general position set. If $x,y \in S$, then no $x_i$ and no $y_i$ can be in $S$ and $S$ can contain at most one $w_j$; therefore $S$ induces a triangle on vertices $x,y,w_j$. Assume that $x \in S, y \notin S$ (the case $y\in S$, $x \notin S$, is obviously symmetrical). The sets $\{ x,x_i\} $ are maximal general position sets, so we can assume that $S \cap X = \emptyset $; then, noting that $S$ can contain at most one vertex of $W$, we must have $S = \{ x,w_j\} \cup Y$ for some $w_j$. Finally, we can assume that $x,y \not \in S$. Suppose that for some $1 \leq i\leq r$ we have $x_i,y_i \in S$; then $S$ cannot contain any further vertices from $X \cup Y$ and we must have $S = \{ x_i,y_i\} \cup W$. Otherwise, $S$ consists of one vertex from each set $\{ x_i,y_i\} $ along with the vertices of $W$.
		
		We conclude that each maximal general position set in $G(r,s)$ has one of the following types: 
		\begin{align*}
			A: & \quad \{ x,x_i\}\ {\rm or}\ \{ y,y_i\}; \\
			B: & \quad \{ x,y,w_j\} ; \\
			C: & \quad \{ x_i,y_i\} \cup W; \\
			D: & \quad \{ x,w_j\} \cup Y\ {\rm or}\ \{ y,w_j\} \cup X; \\
			E: & \quad S\ {\rm consists\ of}\ W\ {\rm together\ with\ one\ vertex\ from\ each\ set}\ \{ x_i,y_i\}. 
		\end{align*}
		These sets have orders $A$: 2; $B$: 3; $C$: $s+2$; $D$: $r+2$; and $E$:  $r+s$.  As $r > s \geq 3$, the desirability of these sets to Builder is given by $E>D>C>B>A$.
		
		Let us first examine the B-game on $G(r,s)$. If Builder's first move is in $\{ x,y\} \cup X \cup Y$, then Blocker can finish the game on their first move by completing the set to a maximal general position set of Type $A$. Therefore we can assume that Builder starts at $w_1$. Blocker has a strategy to restrict Builder to a set of order at most $r+2$ by now choosing $x$ as their first move, forcing a set of Type $D$ or $B$. We show that Builder can guarantee a set of order at least $r+2$. If Blocker's first move is in $\{ x,y\} $ (again, assume this to be $x$), then Builder can reply with $y_1$ to force a set of Type $D$, whereas if Blocker's first move is in $X \cup Y$, say $x_1$, then Builder can reply with $y_2$ to force a set of Type $E$. Suppose now that Blocker takes their first move in $W$; now the outcome is restricted to a set of Type either $E$ or $C$. Now as $s$ is odd, Builder can force Blocker to move in $X \cup Y$ by replying to each move of Blocker in $W$ by choosing another vertex of $W$; by symmetry, we can suppose that Blocker eventually chooses the vertex $x_1$, at which point Builder can reply with $y_2$, forcing a set of Type $E$. It follows that if both players use optimal strategies, the outcome will be a set of Type $D$ with $r+2$ vertices.
		
		Now we examine the B'-game. If Blocker's first move is in $\{ x,y\} $, say $x$, then the best Builder can hope for is a set of Type $D$, which they can force by choosing $y_1$. If Blocker's first move is in $X \cup Y$, say at $x_1$, then Builder could force a set of Type $E$ by replying with $y_2$.

		We now show that if Blocker starts at a vertex of $W$, then they have a strategy to limit Builder to a set of order $s+2$, which Builder can achieve. Suppose then that Blocker chooses vertex $w_1$ as their first move. It would be a bad choice for Builder to take either $x$ or $y$ as their first move, since Blocker could reply with the remaining vertex of $\{ x,y\} $ and limit Builder to a set of Type $B$. However, Builder can achieve at least a set of Type $C$ by taking vertex $x_1$ as their first move. It can be seen that Blocker can limit Builder to a set with at most $s+2$ vertices as follows. As $s$ is odd, if Blocker replies to any move of Builder in $W$ by selecting another vertex of $W$, then Builder is forced to take a move in $\{ x,y\} \cup X \cup Y$; a move in $\{ x,y\} $ is either impossible for Builder (if at least two vertices of $W$ have already been selected) or unwise as a first move for Builder, as already discussed, so we can assume that Builder's first move outside of $W$ is $x_1$. Blocker replies to Builder's move $x_1$ with $y_2$, which forces a set of Type $C$ with $s+2$ vertices. Therefore the best possible strategy for Blocker gives a set of order $s+2$.
		
		It follows that $\gpg (G(r,s)) = r+2$ and $\gpg '(G(r,s)) = s+2$. As $r$ can be arbitrarily larger than $s$, this proves the result.
	\end{proof}
	
	We remark that if we allow $s = 1$ then a simple argument along the lines of the above proof shows that the graph $G(r,1)$ has $\gpg (G(r,1)) = r+2$ and $\gpg '(G(r,1)) = 3$. 
	
	Notice that for all of the graphs constructed in Theorem~\ref{thm:largerbuildernum} the number $\gpg '(G)$ is odd. This raises the question: for which pairs $(a,b)$ is there a graph with $\gpg (G) = a$ and $\gpg '(G) = b$? If there does exist such a graph, we will say that the pair $(a,b)$ is \emph{realisable}. By Proposition~\ref{prop:small builder num} we know that $(a,a+r)$ is realisable for $a \geq 2$ and $r \geq 0$. Trivially $(1,1)$ is realisable, but for $a \geq 2$ the pairs $(1,a)$ and $(a,1)$ are not realisable. By Corollary~\ref{cor:realisation for two} a pair $(a,2)$ is realisable if and only if $a = 2$. Finally, Theorem~\ref{thm:largerbuildernum} shows that any pair $(a+r,a)$ with $a \geq 3$ and $r \geq 0$ is realisable. It remains only to settle the realisability of pairs $(a,b)$ with $a > b$ and $b$ even. We now show that some such pairs are realisable.
	
	\begin{theorem}
		The pairs $(j+k+1,j+2)$ are realisable for $j \geq k \geq 1$.
	\end{theorem}
	\begin{proof}
		We define the graph $H(j,k)$ as follows, where $j \geq k \geq 1$. Take a clique of order $2(j+k)$ and partition it into four sets, $J_1$, $J_2$, $K_1$ and $K_2$, where $|J_1| = |J_2| = j$ and $|K_1| = |K_2| = k$. Add three new vertices $z,x_1$ and $x_2$. Join $x_i$ to all vertices of $J_i \cup K_i$ for $i = 1,2$ and finally join $z$ to every vertex of $J_1 \cup J_2$.
		
		Consider the B'-game on $H(j,k)$. Blocker has a strategy to limit Builder to a set of at most $j+2$ vertices. Blocker chooses $z$ as their first move. Now if Builder takes their next move in $K_1$, Blocker replies with $x_1$, and likewise if Builder's first move is in $K_2$, Blocker replies with $x_2$. If Builder's first move is in $J_1$, Blocker takes $x_2$ as their second move and if Builder moves in $J_2$, Blocker chooses $x_1$. Finally, if Builder chooses $x_1$ or $x_2$, Blocker can end the game by choosing the remaining vertex in $\{ x_1,x_2\} $. Conversely, Builder can guarantee a set of order at least $j+2$ by taking their first move in $J_1 \cup J_2$.  
		
		Now we deal with the B-game. When Builder goes first, they have a strategy to construct a set of order at least $1+j+k$. Builder chooses a vertex of $J_1$ as their first move. If Blocker moves in $\{ z\} \cup J_1$, Builder replies with a vertex in $J_2$, which, since the smallest maximal general position set containing $J_1 \cup J_2$ is $\{ z\} \cup J_1 \cup J_2$, guarantees a set of order at least $2j+1$. If Blocker moves in $J_2$, Builder replies with $z$. If Blocker moves in $K_1 \cup K_2$, Builder can guarantee that the whole clique of order $2(j+k)$ is selected by taking their second move in $J_2$. Finally, if Blocker chooses their first vertex in $\{ x_1,x_2\} $, Builder replies with a vertex of $K_1$ to guarantee at least $1+j+k$ vertices in the resulting set. However, Blocker can limit Builder to at most $1+j+k$ vertices, since if Builder moves in $J_i \cup K_i$ for $i = 1,2$, Blocker can reply with $x_i$, whilst if Builder chooses an $x_i$, $i = 1,2$, for their first move, then Blocker can choose a vertex of $J_i$ for their first move and, finally, if Builder selects $z$ for their first move, then Blocker can restrict Builder to a set of most $j+2$ vertices by replying with $x_1$. 
		
		It follows that $\gpg (H(j,k)) = 1+j+k$ and $\gpg '(H(j,k)) = j+2$. 
	\end{proof}
	
	We conclude this section with the following open problem.
	
	\begin{problem}\label{prob:realisable pairs}
		Is every pair $(a,b)$ with $a > b > 2$ and even $b$ realisable?
	\end{problem}
	
	\section{The B'-game played on trees}\label{sec:trees}
	
	In this section we have a closer look to the B'-game played on trees. (There is no need to consider the B-game due to Corollary~\ref{cor:bipartite}.) We first prove the following upper bound and later, in Theorem~\ref{thm:trees-characterization} characterise the trees that attain the bound. 
	
	\begin{theorem}
		\label{thm:trees-upper bound}
		If $T$ is a tree, then $\gpg'(T)\le \ell(T) - \Delta(T) + 2$. 
	\end{theorem}
	
	\begin{proof}
		Consider the B'-game played on $T$. We need to show that Blocker has a strategy which guarantees that no matter how Builder plays, no more than $\ell(T) - \Delta(T) + 2$ vertices will be selected during the game. The strategy of Blocker is to start the game on a vertex $b_1'$, where $\deg_T(b_1') = \Delta(T)$. Let $T_i$, $i\in [\Delta(T)]$, be the components of $T-b_1'$ and assume without loss of generality that the first move $b_1$ of Builder is in $T_1$. Then all the remaining moves of both players in the rest of the game lie in $T_1$. Indeed, if some later move, say $u$, lies in $T_i$, $i > 1$, then $b_1'$ lies in the $u,b_1$-geodesic. 
		
		Now let $T_1'$ be the subtree of $T$ induced by $V(T_1) \cup \{b_1'\}$. Clearly, each of the subtrees $T_i$, $i\ge 2$, contains at least one leaf of $T$. Note also that $b_1'$ is a leaf of $T_1'$ but it is not a leaf of $T$. It follows that $\ell(T_1') \le \ell(T) - (\Delta(T) - 1) + 1$ (where ``$+1$" comes from the leaf $b_1'$ of $T_1'$). Since our game is restricted to the vertices of $T_1'$, we thus have
		$$\gpg'(T) \le \gp(T_1') = \ell(T_1') \le \ell(T) - (\Delta(T) - 1) + 1\,,$$
		where we have used the fact that for any tree the general position number equals the number of its leaves~\cite[Corollary 3.7]{Manuel-2018}. 
	\end{proof}
	
	To characterise the graphs that attain the equality in Theorem~\ref{thm:trees-upper bound}, consider the following trees. If $k\ge 2$, and $t_i$, $i\in [k]$, are non-negative integers, then let $T_{t_1,\dots,t_k}$ denote the tree obtained from the path of order $k$, to be called the {\em central path} of $T_{t_1,\dots,t_k}$, by respectively attaching $t_1,\dots,t_k$ leaves to its consecutive vertices. In particular, $T_{0,\dots,0} \cong P_k$ and $T_{1,0,\dots,0,1} \cong P_{k+2}$. 
	
	\begin{proposition}
		\label{prop:gpg-Tk}
		If $k\geq 2$, $t_1\ge 1$, and $t_k\ge 1$, then  $$\gpg'(T_{t_1,\dots,t_k})=\min\limits_{i\in [k]}\max\{\sum\limits_{j=1}^{i-1}t_j,\sum\limits_{j=i+1}^{k}t_j\}+1.$$
	\end{proposition}
	
	\begin{proof}
		Let $k\geq 2$, and set $T = T_{t_1,\dots,t_k}$ for the rest of the proof. Let $P$ be the central path of $T$ and let $u_1, \ldots, u_k$ be its consecutive vertices. Consider the B'-game and let Blocker start by the move $x$. Based on the position of $x$ in $T$, we consider the following cases.
		
		\medskip\noindent
		{\bf Case 1.} $x=u_i$, $i\in[k]$.\\
		If Builder replies with a vertex adjacent to $u_i$, then the game is over. Hence we consider the following two subcases. 
		
		Assume first that Builder starts with the move $u_s$, where $u_s$ is not adjacent to $u_i$. In this subcase, every vertex played in the rest of the game is a leaf adjacent to an internal vertex of the $u_i,u_s$-path. Moreover, by Lemma~\ref{lem:useful} all these vertices will be played. It follows that if $s \le i-2$, the game lasts $2+\sum_{j=s+1}^{i-1}t_j$ moves. Since $t_1\ge 1$, the game thus lasts at most $1+\sum_{j=1}^{i-1}t_j$ moves. Similarly, if $s \ge i+2$, then the game lasts $2+\sum_{j=i+1}^{s-1}t_j$ moves and because $t_k\ge 1$, the game thus lasts at most $1+\sum_{j=i+1}^{k}t_j$ moves.
		
		Assume second that Builder starts with the move $y$, where $y$ is a leaf adjacent to $u_s$, where $s\ne i$. Then each of the moves in the rest of the game is not an internal vertex of the $x,y$-path in $T$. If $s<i$, then each vertex in $N[u_\ell]$, where $\ell \ge i$, is on a common shortest path with $x$ and $y$. It follows that at most $\sum_{j=1}^{i-1}t_j+1$ moves will be played by the end of the game when $s<i$. Similarly, if $i<s$, then at most $\sum_{j=i+1}^{k}t_j+1$ moves will be played.
		
		\medskip\noindent
		{\bf Case 2.} $x$ is a leaf attached to $u_i$, $i\in[k]$. \\
		In this case, Builder will not reply with the vertex $u_i$, because in that case then game would be over. 
		So we consider the following subcases.
		
		Assume first that Builder starts with the move $u_{s}$, where $s\neq i$. In this subcase, the moves in the rest of the game can not lie on the $u_i,u_s$-path in $T$.  Assume that $s<i$. Then in the rest of the game at most one vertex from the $u_{s+1},u_k$-path will be played. Moreover, if such a vertex will be played, then the leaves attached to $u_k$ will not be played, and we have assumed that there is at least one such leaf. It follows that the game lasts at most  $\sum_{j=s+1}^{k}t_j+1$ moves. By a parallel argument, the game lasts at most $\sum_{j=1}^{s-1}t_j+1$ moves if $s>i$.
		
		Assume second that Builder starts by selecting a leaf attached to $u_s$, where $s\in [k]$. In this subcase, by Lemma~\ref{lem:useful}, all the leaves of $T$ can  be played in the rest of the game. Hence  at most $\sum_{j=1}^{k}t_j$ moves will be selected in the game.
		
		We have proved by now that the B'-game will last at most 
		$$\gpg'(T)=\min\limits_{i\in [k]}\max\{\sum\limits_{j=1}^{i-1}t_j,\sum\limits_{j=i+1}^{k}t_j\}+1$$
		moves. 
		
		Now let $i\in [k]$ be selected such that $\max\{\sum\limits_{j=1}^{i-1}t_j,\sum\limits_{j=i+1}^{k}t_j\} $ is minimised over all $i$. Let Blocker start the game by playing $u_i$. If $\sum\limits_{j=1}^{i-1}t_j \ge \sum\limits_{j=i+1}^{k}t_j$, then Builder replies by playing an arbitrary leaf attached to $u_1$. By Lemma~\ref{lem:useful}, by the end of the game exactly all the leaves attached to $u_1, \ldots, u_{i-1}$ will be played, so that the game will last at least  $\sum\limits_{j=1}^{i-1}t_j + 1$ moves. Similarly, if $\sum\limits_{j=1}^{i-1}t_j < \sum\limits_{j=i+1}^{k}t_j$, then Builder replies by playing an arbitrary leaf attached to $u_k$ and then Lemma~\ref{lem:useful} implies that the game will last at least  $\sum\limits_{j=i+1}^{k}t_j + 1$ moves. We conclude that the game lasts at least $\max\{\sum\limits_{j=1}^{i-1}t_j,\sum\limits_{j=i+1}^{k}t_j\}+1$ moves.  
	\end{proof}
	
	In view of Proposition~\ref{prop:gpg-Tk} we now set
	$${\cal T}' = \{T_{t_1,\ldots, t_k}:\ k\geq 2,\ t_1\ge 1,\ t_k\geq \sum_{i=1}^{k-1}t_i\}\,$$
	and state the following: 
	
	\begin{proposition}
		\label{prop:equality-for-T'}
		If $T\in {\cal T}'$, then $\gpg'(T) = \ell(T) - \Delta(T) + 2$.
	\end{proposition}
	
	\begin{proof}
		Let $u_1,\ldots, u_k$ be the consecutive vertices of the $P_k$, to which $t_1,\ldots,t_k$ leaves are  respectively attached. By the definition of ${\cal T}'$ we have $t_1\geq 1$ and $t_k\geq 1$. 
		
		Assume first that $t_2=\cdots=t_{k-1}=0$. If $t_1=t_k=1$, then $T$ is a path and $\gpg'(T)=2$ holds. If $t_1=1$ and $t_k\geq 2$, then after Blocker first plays $u_k$, the game is over after an arbitrary first move of Builder. Hence $\gpg'(T) = 2$ in this case. Assume next that $t_1\geq 2$. Then $\deg_T(u_1)=t_1+1\ge 3$ and $\deg_T(u_k)= t_k+1\ge 3$. Since $\ell(T)=t_1+t_k$ and $t_k\geq t_1$, by Proposition~\ref{prop:gpg-Tk},  $$\gpg'(T)=t_1+1=\ell(T)-t_k+1=\ell(T)-\Delta(T)+2.$$
		
		Assume second that $t_i\geq 1$ for some $2\leq i\leq k-1$. If $t_1 = 1$, then we can argue similarly as above that $u_k$ is an optimal first move of Blocker which gives us the required conclusion. Assume next that $t_1\ge 2$. Hence $u_k$ is the vertex with maximum degree in $T$ and by our assumption we have $\deg_T(u_k) = t_k + 1 \geq \sum_{j=1}^{k-1}t_j + 1$. By Proposition~\ref{prop:gpg-Tk} we have 
		$$\gpg'(T) = \min\limits_{i\in [k]}\max\{\sum\limits_{j=1}^{i-1}t_j,\sum\limits_{j=i+1}^{k}t_j\}+1.$$
		Because $t_k\geq \sum_{j=1}^{k-1}t_j$ we infer that $$\min\limits_{i\in [k]}\max\{\sum\limits_{j=1}^{i-1}t_j,\sum\limits_{j=i+1}^{k}t_j\} = \sum_{j=1}^{k-1}t_j\,.$$
		Consequently, 
		\begin{align*}
			\gpg'(T) & = \sum_{j=1}^{k-1}t_j+1 \\
			& = (\ell(T) - t_k) + 1= (\ell(T) - (\Delta(T) - 1)) + 1  \\
			& = \ell(T)-\Delta(T)+2\,,
		\end{align*}
		and we are done.
	\end{proof}
	
	Now let ${\cal T}$ be the family of trees that contains all the stars $K_{1,n}$, $n\ge 2$, and all the trees that can be obtained from some tree $T\in {\cal T}'$ by subdividing each of the edges of $T$ an arbitrary number of times. An example can be seen in Fig.~\ref{fig:graph-Tk} where the tree from ${\cal T}$ is obtained from $T_{k,\ell}\in {\cal T}'$ by subdividing some of its edges.   
	
	\begin{figure}[ht!]
		\begin{center}
			\begin{tikzpicture}[scale=1.2,style=thick]
				\tikzstyle{every node}=[draw=none,fill=none]
				\def\vr{3pt} 
				\begin{scope}[yshift = 0cm, xshift = 0cm]
					\path (-1,0) coordinate (x1);
					\path (0,0) coordinate (x2);
					\path (2.5,0) coordinate (x3);
					\path (3.5,0) coordinate (x4);
					\path (-1,1) coordinate (x5);
					\path (-1,-1.5) coordinate (x6);
					\path (3.5,1) coordinate (x7);
					\path (3.5,-1.5) coordinate (x8);
					\path (-0.67,-1) coordinate (y2);
					\path (-0.35,-0.5) coordinate (y3);
					\path (-0.5,0.5) coordinate (y4);
					\path (3.00,0) coordinate (y5);
					\path (0.65,0) coordinate (y1);
					\path (1.3,0) coordinate (y6);
					\path (1.95,0) coordinate (y7);
					\draw (x1) -- (x2) -- (x3) -- (x4);
					\draw (x2) -- (x5);
					\draw (x2) -- (x6);
					\draw (x3) -- (x8);
					\draw (x3) -- (x7);
					
					\draw (x1)  [fill=white] circle (\vr);
					\draw (x2)  [fill=white] circle (\vr);
					\draw (x3)  [fill=white] circle (\vr);
					\draw (x4)  [fill=white] circle (\vr);
					\draw (x5)  [fill=white] circle (\vr);
					\draw (x6)  [fill=white] circle (\vr);
					\draw (x7)  [fill=white] circle (\vr);
					\draw (x8)  [fill=white] circle (\vr);
					\draw (y1)  [fill=white] circle (\vr);
					\draw (y2)  [fill=white] circle (\vr);
					\draw (y3)  [fill=white] circle (\vr);
					\draw (y4)  [fill=white] circle (\vr);
					\draw (y5)  [fill=white] circle (\vr);
					\draw (y6)  [fill=white] circle (\vr);
					\draw (y7)  [fill=white] circle (\vr);
					\node at (-1,-0.6) {$\vdots$};
					\node at (3.5,-0.6) {$\vdots$};
					\node [below=0.5mm] at (0.2,0) {$u_1$};
					\node [below=0.5mm] at (2.4,0) {$u_2$};
					\draw [decorate,decoration={brace,amplitude=10pt,raise=4pt}] (-1.3,-1.8)-- (-1.3,1.3)node[left=15pt,midway]{$k$};
					\draw [decorate,decoration={brace,amplitude=10pt,raise=4pt, mirror}] (3.7,-1.8)-- (3.7,1.3)node[right=15pt,midway]{$\ell$};
				\end{scope}
			\end{tikzpicture}'
		\end{center}
		\caption{A tree from the family ${\cal T}$}\label{fig:graph-Tk}
	\end{figure}
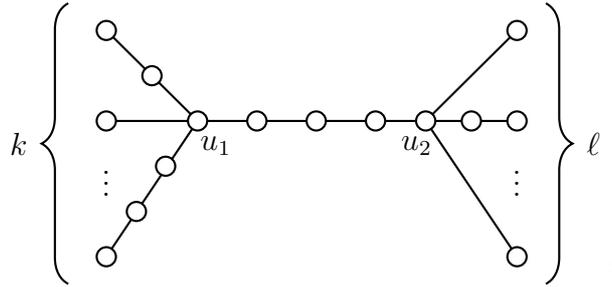

	We can now characterise the equality case in Theorem~\ref{thm:trees-upper bound}.

	\begin{theorem}
		\label{thm:trees-characterization}
		Let $T$ be a tree of order at least $3$. Then $\gpg'(T)= \ell(T) - \Delta(T) + 2$ if and only if $T\in {\cal T}$. 
	\end{theorem}
	
	\begin{proof}
		Suppose first that $T\in {\cal T}$. If $T\cong K_{1,n}$, then $\gpg'(T) = 2 = \ell(T) - n + 2$. Otherwise $T$ is obtained from some tree $T'\in {\cal T}'$ by subdividing some of its edges. By Proposition~\ref{prop:equality-for-T'}, $\gpg'(T')=\ell(T')-\Delta(T')+2$. Moreover,  subdividing edges of $T'$ does not change the arguments of the proof of Proposition~\ref{prop:equality-for-T'} applied to $T$, that is, we also have $\gpg'(T)=\ell(T)-\Delta(T)+2$.
		
		Conversely, assume that $\gpg'(T)= \ell(T) - \Delta(T) + 2$ and consider the B'-game played on $T$. Suppose that $T$ contains $k$ vertices of degree at least $3$. If $k=0$, then $T$ is a path of order at least $3$ and its belongs to ${\cal T}$. If $k = 1$, then $T$ is a starlike tree and it also belongs to ${\cal T}$. The same conclusion holds if $k = 2$. Hence suppose in the rest that $k\ge 3$ and let $u_1,\ldots,u_k$ be the vertices of degree at least $3$ in $T$.  
		
		Set $\Delta = \Delta(T)$ and assume without loss of generality that  $\deg_T(u_k)=\Delta$. We first claim that $u_k$ does not lie on the $u_i,u_j$-path in $T$, where $i\neq k$ and $j\neq k$. Suppose this is not the case. Let Blocker start the game by the vertex $u_k$. For any $\ell\in[\Delta]$, let $T'_{\ell}$ be the components of $T-u_k$, and let $T_\ell$ be the subtree of $T$ induced by $V(T'_\ell)\cup \{u_k\}$. Assume without loss of generality that Builder's first move is in $T_1$.  Then all the moves in the rest of the game lie in $T_1$. Since $u_k$ lies on the $u_i,u_j$-path, at least one of $u_i$ and $u_j$ does not belong to $T_1$, say $u_i\in T_2$. As $\deg_T(u_i)\ge 3$, we see that $T_2$ has at least two leaves of $T$. It follows that 
		$$\ell(T_1)\leq \ell(T)- [(\Delta-2) + 1] = \ell(T)-\Delta(T) + 1\,.$$
		Using~\eqref{eq:for-B'-game} we get
		$$\gpg'(T) \le \gpg'(T_1) \le \ell(T_1) \leq \ell(T)-\Delta(T)+1\,.$$ 
		This contradiction proves the claim. 
		
		By the just proved claim we may assume without loss of generality that the vertices $u_1,\ldots, u_{k-1}$ lie in $T_1$. We next claim that for any two indices $i,j\ne k$, the vertices $u_i$, $u_j$, and $u_k$ lie on a common path in $T_1$ (where, clearly, $u_k$ is one of its endvertices). Suppose on the contrary that this is not the case. Then as the second move, Blocker can select one of $u_i$ and $u_j$, say $u_i$. Since $\deg_T(u_i) \ge 3$, there are at least two leaves of $T_1$ different from $u_k$ which cannot be played in the rest of the game. It follows that the game lasts at most $\ell(T)-(\Delta-1)-2+2=\ell(T)-\Delta + 1$ moves, a contradiction to our assumption. 
		
		We have thus proved that for any two indices $i,j\ne k$, the vertices $u_i$, $u_j$, and $u_k$ lie on a common path in $T_1$. This in turn implies that all the vertices $u_1, \ldots, u_k$ lie on a common path $P$ in $T$, where $u_k$ is one of its endvertices and we may also assume that $u_1$ is the other endvertex of $P$. Hence, $T$ contains the path $P$ and some pendant paths attached to the vertices $u_1, \ldots, u_k$, where at least one pendant path is attached to $u_1$ and at least one to $u_k$. Let $t_i$ be the number of pendant paths attached to $u_i$, $i\in [k]$. Hence $T$ is obtained from $T_{t_1,\ldots, t_k}$ by subdividing some of its edges. To complete our proof that $T\in {\cal T}$, we thus need to show that $t_k\geq \sum_{i=1}^{k-1}t_i$. 
		
		Suppose on the contrary that $t_k\leq \sum_{i=1}^{k-1}t_i-1$. Since $T$ is obtained from $T_{t_1,\ldots, t_k}$ by subdividing some of its edges, we can argue as above that Proposition~\ref{prop:gpg-Tk} applies to $T$, that is,  
		$$\gpg'(T) = \min\limits_{j\in [k]}\max\{\sum\limits_{i=1}^{j-1}t_i,\sum\limits_{i=j+1}^{k}t_i\}+1\,.$$
		Considering $j=k-1$, we get $\gpg'(T)\leq \max\{\sum\limits_{i=1}^{k-2}t_i,t_k\}+1$. 
		As we have assumed that $t_k\leq \sum_{i=1}^{k-1}t_i-1$, it follows that 
		\begin{equation}
			\label{eq:first}    
			t_k+1\leq \sum_{i=1}^{k-1}t_i=\ell(T)-t_k=\ell(T)-(\Delta-1)=\ell(T)-\Delta+1\,.
		\end{equation}
		Since $\deg_T(u_{k-1})\geq 3$, then 
		\begin{align}
			\sum_{i=1}^{k-2}t_i+1 & = \ell(T)-t_k-t_{k-1}+1 
			= \ell(T)-(\Delta-1)-t_{k-1}+1 \nonumber \\
			& \leq \ell(T)-(\Delta-1)-1+1 =\ell(T)-\Delta+1\,. \label{eq:second}
		\end{align}
		By~\eqref{eq:first} and~\eqref{eq:second} we get $\max\{\sum_{i=1}^{k-2}t_i,t_k\}+1\leq \ell(T)-\Delta + 1$ which in turn implies that $\gpg'(T)\leq \ell(T)-\Delta+1$. This contradiction to our assumption implies that  $t_k\geq \sum_{i=1}^{k-1}t_i$ and we conclude that $T\in {\cal T}$.
	\end{proof}
	
	\section*{Concluding remarks}
	
	In this paper we have introduced the Builder-Blocker general position games and derived several of their properties. We conclude with some open problems that are suggested by this research.
	
	Firstly, we suggest studying the complexity of the decision version of Builder-Blocker general position games. On the basis of the hardness of the avoidment game~\cite{ullas-2023+} we make the following conjecture.
	
	\begin{conjecture}
		The decision version of the B- and B'-games are PSPACE-complete.
	\end{conjecture}
	
	We know from Inequalities~\eqref{eq:for-B-game} and~\eqref{eq:for-B'-game} that the game general position number of a graph $G$ lies between $\gp ^-(G)$ and $\gp (G)$. Is it possible to characterise the case in which equality is achieved in any of these bounds? Or is this a hard problem? 
	\begin{problem}
		Characterise graphs for which equality holds in the upper or lower bounds in Inequalities~\eqref{eq:for-B-game} or~\eqref{eq:for-B'-game}.
	\end{problem}

	In Lemma~\ref{lem:universal-line} and in Theorem~\ref{thm:gpg'=2} we have characterised the family of graphs with $\gpg(G) = 2$ and the family of graphs with $\gpg'(G) = 2$. The following problem represents the next step in this direction.
	
	\begin{problem}
		Characterise graphs with $\gpg(G) = 3$ or $\gpg'(G) = 3$, or the family of graphs with order $n$ and $\gpg (G) = n-j$ or $\gpg' (G) = n-j$ for small $j$.
	\end{problem}
	
	In Theorem~\ref{thm:Kneser} we determined $\gpg$ and $\gpg'$ for the Kneser graphs $K(n,2)$. This could be developed further as follows.
	
	\begin{problem}
		Determine the invariants also for the Kneser graphs $K(n,k)$, where $k > 2$, and Johnson graphs $J(n,k)$.    
	\end{problem}
	
	In Problem~\ref{prob:realisable pairs} we asked whether all pairs $(a,b)$ with $a,b \geq 2$ are realisable. For those pairs $(a,b)$ that we have already proven to be realisable, it would also be of interest to find the \emph{smallest} graphs with these parameters. We also suggest investigating game versions of other position type problems, such as the monophonic position problem~\cite{thochatuistef}.
	
	\section*{Acknowledgements}
	
	Sandi Klav\v{z}ar acknowledges the financial support from the Slovenian Research Agency (research core funding P1-0297 and projects N1-0218 and N1-0285). The research of James Tuite was partially funded by LMS Research in Pairs grant number 42235; he also thanks the University of Ljubljana for their hospitality. The authors are also grateful to Grahame Erskine for assistance with computational verification of these results, Aoise Evans for discussion of the manuscript, and the three anonymous reviewers for their helpful comments.
	

\end{document}